\documentclass[11pt]{amsart}

\usepackage{geometry, amsmath, amssymb, amsthm, mathrsfs, setspace, comment, amscd, latexsym, mathtools, bm}
\usepackage[all]{xy}                          
\usepackage{parskip}
\usepackage[abs]{overpic} 
\usepackage{pict2e}

\usepackage{xcolor,graphicx}

\usepackage{hyperref}
\hypersetup{
    colorlinks=true,
    citecolor=blue,
    linkcolor=blue,
    urlcolor=blue,
}

\makeatletter
\@addtoreset{equation}{section}

\makeatother

\newtheorem{theorem}{Theorem}[section]
\newtheorem{lemma}[theorem]{Lemma}

\newtheorem{proposition}[theorem]{Proposition}

\theoremstyle{definition}
\newtheorem{definition}[theorem]{Definition}
\newtheorem{example}[theorem]{Example}
\newtheorem{question}[theorem]{Question}

\newtheorem{remark}[theorem]{Remark}

\newcommand{\del}{\partial}
\newcommand{\N}{\mathbb{N}}
\newcommand{\Z}{\mathbb{Z}}
\newcommand{\R}{\mathbb{R}}

\newcommand{\C}{\mathbb{C}}
\newcommand{\CP}{\mathbb{CP}}

\DeclareMathOperator{\FS}{FS}

\DeclareMathOperator{\im}{\mathrm{Im}\, }
\DeclareMathOperator{\PD}{PD}

\DeclareMathOperator{\orb}{orb}

\title[Stein fillability of $S^1$-bundles over symplectic manifolds]{A note on Stein fillability of circle bundles over symplectic manifolds}
\author[Takahiro Oba]{Takahiro Oba}
\address{Department of Mathematics, Osaka University, Toyonaka, Osaka 560-0043, Japan}
\email{taka.oba@math.sci.osaka-u.ac.jp}
\date{\today}

\begin{document}

\maketitle

\begin{abstract}
We show that, given a closed integral symplectic manifold $(\Sigma, \omega)$ of dimension $2n \geq 4$,  for every integer $k>\int_{\Sigma}\omega^{n}$, the Boothby--Wang bundle over $(\Sigma, k\omega)$ carries no Stein fillable contact structure. 
This negatively answers a question raised by Eliashberg. 
A similar result holds for Boothby--Wang orbibundles. 
As an application, we prove the non-smoothability of some isolated singularities. 
\end{abstract}


\section{Introduction}\label{section: intro}

The study of contact topology has revealed both flexible and rigid aspects of contact structures. 
According to Eliashberg \cite{El_OT} and Borman--Eliashberg--Murphy \cite{BEM}, the existence and classification of overtwisted contact structures essentially reduce to algebraic topology; 
as a result, every almost contact structure is realized as an overtwisted contact structure. 
In contrast, there exist almost contact structures that cannot be represented by tight contact structures (e.g. \cite{EH} and \cite{LS}). 
This implies that tight contact structures are more rigid. 
Our primary interest lies in \textit{Stein fillable} contact structures, which form a special class of tight contact structures (see Section \ref{section: Stein} for the definition). 

Circle bundles over symplectic manifolds provide important examples of contact manifolds. 
Let $(\Sigma, \omega)$ be a closed integral symplectic manifold of dimension $2n$, i.e., $[\omega] \in H^{2}(\Sigma; \R)$ has an integral lift. 
(Throughout the paper, we always assume $\Sigma$ to be oriented so that $\omega^{n}$ is a positively oriented volume form and an integral lift of $[\omega]$ to be fixed.) 
The \textit{Boothby--Wang bundle} over $(\Sigma, \omega)$ is a principal circle bundle $p \colon M \rightarrow \Sigma$ with Euler class $e(M)=-[\omega]$, which is uniquely determined up to isomorphism. 
A connection $1$-form on $M$ with curvature form $2\pi \omega$ defines a contact structure $\xi$ on $M$. 
We call $(M, \xi)$ the \textit{Boothby--Wang contact manifold} associated with $(\Sigma, \omega)$.  
Boothby--Wang bundles naturally appear in the study of contact and symplectic topology. 
For example, they are related to symplectic submanifolds constructed by Donaldson \cite{Don}, as explained in \cite{DL} and \cite{Giroux_Remarks} (see also Remark \ref{rmk: verification}). 

From the fillability viewpoint, every Boothby--Wang contact manifold $(M,\xi)$ is strongly fillable since the disk bundle associated with $p \colon M \rightarrow  \Sigma$ serves as a strong symplectic filling of $(M,\xi)$. 
This leads to the following question, as posed in \cite[Section 4.1.2]{Biran_ECM} and \cite[Problem 6.6]{BCS}: 
Is the contact manifold $(M,\xi)$ Stein fillable?
More generally, one may ask whether the Boothby--Wang bundle $M$ admits a Stein fillable contact structure or not. 
There are examples of Boothby--Wang bundles admitting no Stein fillable contact structures (e.g. \cite{EKP} and \cite{PP}); see also \cite{KO} for non-Stein fillable Boothby--Wang contact manifolds. 
Furthermore, Bowden--Crowley--Stipsicz \cite{BCS} provided a characterization of an almost contact structure that is represented by a Stein fillable contact structure. 
However, applying this characterization to a specific case is yet challenging in practice, as it is expressed in terms of cobordism theory.

The main theorem in this paper, stated below, gives a new class of Boothby--Wang bundles without Stein fillable contact structures. 
In particular, this negatively answers the question raised by Eliashberg \cite[p.~4]{CourteAIM}. 

\begin{theorem}\label{thm: main}
Let $(\Sigma, \omega)$ be a closed integral symplectic manifold of dimension $2n \geq 4$. 
Then, for every integer $k> \int_\Sigma \omega^{n}$, the Boothby--Wang bundle $M_k$ over the symplectic manifold $(\Sigma, k\omega)$ carries no Stein fillable contact structure. 
\end{theorem}

The proof of this theorem is inspired by a topological argument in \cite{PP}. 
What we really show is that the Boothby--Wang bundle $M_k$ with $k > \int_\Sigma \omega^{n}$ does not bound a compact manifold with the homotopy type of a CW complex of dimension $ \leq n+1$ (see Proposition~\ref{prop}). 
Together with a result on the homotopy types of Stein domains, this completes the proof of the theorem in Section \ref{section: proof}. 

Here are some remarks on the above theorem. 
Suppose that $\Sigma$ is a closed Riemann surface with an integral symplectic form $\omega$. 
The Boothby--Wang bundle $M_k$ over $(\Sigma, k\omega)$ carries a Stein fillable contact structure if $\chi(\Sigma)<k\int_{\Sigma}\omega$, where $\chi(\Sigma)$ denotes the Euler characteristic of $\Sigma$ (see \cite{GS} for example). 
Thus, Theorem~\ref{thm: main} does not hold true for surfaces in general since $\chi(\Sigma)$ is at most $2$. 
We also remark that the condition on $k$ in the above theorem cannot be dropped (see Remark~\ref{rem: k}). 
Additionally, readers who are familiar with approximately holomorphic techniques might wonder at first glance if the theorem contradicts a result of Giroux \cite{Giroux_Remarks}. 
Remark~\ref{rmk: verification} verifies the consistency of the theorem with his result. 

A similar result to Theorem \ref{thm: main} holds for Boothby--Wang orbibundles (see Theorem \ref{thm: orbibundle}). 
Even though the main theorem is involved in the orbibundle case, for simplicity we primarily discuss the usual bundle case, namely the case where the base space of a circle bundle is a manifold. 

In Section~\ref{section: singularity}, we present an application of Theorem~\ref{thm: main} to singularity theory. 
As a result, we establish an obstruction to the smoothability of certain isolated singularities.

It is worth mentioning results on another fillability.
As already observed, every Boothby--Wang contact manifold is strongly fillable.
The real projective space $\mathbb{RP}^{2n+1}$ of dimension $2n+1$ can be regarded as the Boothby--Wang bundle over $(\CP^{n}, 2\omega_{\FS})$, where $\omega_{\FS}$ denotes the Fubini--Study form on $\CP^{n}$. 
In view of this, we equip $\mathbb{RP}^{2n+1}$ with a contact structure $\xi_n$ defined by a connection $1$-form on the circle bundle; 
this in fact can be seen as the reduction of the standard contact structure on $S^{2n+1}$ by the $\Z/2\Z$-action. 
It had been conjectured in \cite{CourteAIM} that $(\mathbb{RP}^{2n+1}, \xi_n)$ ($n \geq 2$) is not exactly fillable.
Zhou \cite{Zhou_RP, Zhou2022fillings} recently showed that some Boothby--Wang contact manifolds of dimension $\geq 5$, which appear as links of quotient singularities, including real projective spaces, are not exactly fillable. 
This brings us the following question. 

\begin{question}
Let $(\Sigma, \omega)$ be a closed integral symplectic manifold of dimension $2n \geq 4$. 
Does the Boothby--Wang bundle over the symplectic manifold $(\Sigma, k\omega)$ carry an exactly fillable contact structure if $k > \int_{\Sigma} \omega^n$? 
\end{question}


\section{Non-Stein fillability of Boothby--Wang bundles}


\subsection{Stein fillable contact manifolds}\label{section: Stein}

We begin by recalling necessary definitions related to Stein fillability. 
The reader is referred to \cite{CE_book} for more details. 

A \textit{Stein domain} is a compact complex manifold $(W,J)$ with boundary that admits a strictly plurisubharmonic function $f \colon W \rightarrow \R$ with $\del W=f^{-1}(\max f)$. 

\begin{definition}
A (cooriented) contact manifold $(M, \xi)$ is said to be \textit{Stein fillable} if there exists a Stein domain $(W,J)$, called a \textit{Stein filling} of $(M,\xi)$, such that $\del W=M$ as oriented manifold and $T\del W \cap JT\del W=\xi$. 
\end{definition}


\subsection{Proof of the main theorem}\label{section: proof}

Now we shall prove Theorem \ref{thm: main}, which immediately follows from the next proposition. 

\begin{proposition}\label{prop}
Let $(\Sigma, \omega)$ be a closed integral symplectic manifold of dimension $2n \geq 4$. 
Then, for every integer $k> \int_\Sigma \omega^{n}$, the Boothby--Wang bundle $M_k$ over the integral symplectic manifold $(\Sigma, k\omega)$ does not bound a compact manifold with the homotopy type of a CW complex of dimension $ \leq n+1$.
\end{proposition}

\begin{proof}[Proof of Theorem \ref{thm: main}] 
Suppose that the Boothby--Wang bundle $M_k$ over the symplectic manifold $(\Sigma, k\omega)$ of dimension $2n \geq 4$ carries a Stein fillable contact structure.
Then, it bounds a Stein domain, which has the homotopy type of a CW complex of dimension $ \leq n+1$ by \cite[Theorem 7.2]{Milnor}.
This however gives a contradiction to Proposition \ref{prop} when $k> \int_\Sigma \omega^{n}$. 
\end{proof}

To show Proposition \ref{prop}, we prove one lemma. 

\begin{lemma}\label{lem: obstruction}
Let $M$ be a closed manifold of dimension $2n+1 \geq 5$. 
Suppose that there exists a second cohomology class $a \in H^2(M;\Z)$ such that $a^{n} \neq 0 \in H^{2n}(M;\Z)$, and moreover, if $\dim M=5$, the class $a$ is a torsion element of $H^2(M;\Z)$. 
Then, $M$ does not bound a compact manifold with the homotopy type of a CW complex of dimension $ \leq n+1$. 
\end{lemma}

\begin{proof}
Suppose, on the contrary, that there is a compact manifold $W$ such that $\del W=M$ and $W$ has the homotopy type of a CW complex of dimension $ \leq n+1$. 
Then, for $n \geq 3$, $H_{2n-1}(W;\Z)=0$, and for $n=2$, $H_3(W;\Z)$ is free. 
In view of the Poincar\'e duality, $H^{3}(W, M;\Z)$ also satisfies the same property as $H_{2n-1}(W;\Z)$. 
Hence, combining this with the exact sequence 
$$
	\cdots \rightarrow H^2(W; \Z) \xrightarrow{i^*} H^2(M; \Z) \xrightarrow{\varphi} H^3(W,M;\Z)  \rightarrow \cdots 
$$
shows the existence of a lift $\tilde{a} \in H^2(W;\Z)$ of the element $a \in H^2(M; \Z)$ with respect to the map $i^*$. 
Here $i$ denotes the inclusion $M=\del W \hookrightarrow W$. 
Consider the following commutative diagram:
$$
   \xymatrix@C=52pt{
    H^2(W; \Z)^{\otimes n} \ar[d]_{} \ar[r]^{(i^*)^{\otimes n}}  & H^2(M; \Z)^{\otimes n} \ar[d] \\
    H^{2n}(W; \Z) \ar[r]_{i^*} & H^{2n}(M;\Z),
   }
$$
where the vertical maps are the cup product. 
This diagram shows that 
$$
	i^*(\tilde{a}^{n})=(i^*\tilde{a})^{n}= a^{n} \neq 0 \in H^{2n}(M;\Z), 
$$
and hence $\tilde{a}^{n} \neq 0 \in H^{2n}(W;\Z)$. 
However, as $2n > n+1=\dim W/2$, we have $H^{2n}(W;\Z)=0$ for any $n\geq 2$, leading to a contradiction. 
\end{proof}

\begin{remark}
For the case $\dim M=2n+1 \geq 7$, i.e., $n \geq 3$, the above lemma can alternatively be obtained from \cite[Theorem 3.1]{PP}.
\end{remark}

\begin{proof}[Proof of Proposition \ref{prop}] 
Let $\pi_k \colon M_k \rightarrow \Sigma$ denote the bundle projection of the Boothby--Wang bundle over $(\Sigma, k\omega)$. 
It suffices to show that $\pi_k^*[\omega]$ meets the assumption on a second cohomology class $a \in H^2(M_k;\Z)$ in Lemma \ref{lem: obstruction} for every integer $k> \int_\Sigma \omega^n$.

We first check that $(\pi_k^*[\omega])^{n} \neq 0 \in H^{2n}(M_k;\Z)$. 
Consider the Gysin exact sequence 
\begin{align}\label{gysin}
	\cdots \rightarrow H^{2n-2}(\Sigma; \Z) \xrightarrow{k[\omega] \smile} H^{2n}(\Sigma; \Z) \xrightarrow{\pi_k^*} H^{2n}(M_k;\Z) \rightarrow \cdots
\end{align}
Since $\omega$ is a symplectic form on the closed manifold $\Sigma$ of dimension $2n$, we have $[\omega^n] \neq 0 \in H^{2n}(\Sigma; \Z)$. 
Identifying $H^{2n}(\Sigma; \Z)$ with $\Z$ via $\alpha \mapsto \alpha([\Sigma])$, 
the image $\im(k[\omega] \smile)$ lies in the subgroup $k\Z \subset H^{2n}(\Sigma; \Z) \cong \Z$. 
This proves that for $k>\omega^{n}([\Sigma])=  \int_\Sigma \omega^{n}$, the cohomology class $[\omega^n]$ is not contained in $\im(k[\omega] \smile)$. 
Therefore, in view of the exactness of the sequence (\ref{gysin}), we have 
$$
	(\pi_k^*[\omega])^{n}=\pi_k^*([\omega^n]) \neq 0 \in H^{2n}(M_k; \Z).
$$ 

What is left is to check that $\pi_k^*[\omega]$ is a torsion for the case $\dim M_k=5$. 
(Note that the proof below actually shows that this is true even if $\dim M_k \geq 5$.)
Again by the Gysin exact sequence 
\begin{align}\label{gysin2}
	H^{0}(\Sigma; \Z) \xrightarrow{k[\omega] \smile} H^{2}(\Sigma; \Z) \xrightarrow{\pi_k^*} H^{2}(M_k;\Z) \rightarrow \cdots,
\end{align}
we see that $k(\pi_k^*[\omega])=\pi^*_k(k[\omega])= 0 \in H^2(M_k;\Z)$. 
This completes the proof. 
\end{proof}

The following example explains that the condition on $k$ in Theorem \ref{thm: main} and Proposition \ref{prop} cannot be dropped. 

\begin{example}\label{rem: k}
Consider $S^2\times S^2$ equipped with the symplectic from $\omega_1+\omega_2$, where $\omega_i$ denotes the pull-back of an area form on $S^2$ with total area $1$ under the projection $S^2 \times S^2 \rightarrow S^2$ to the $i$th factor. 
We claim that for $k=1,2$ the Boothby--Wang contact manifold $(M_{k}, \xi_k)$ associated with $(S^2\times S^2, k(\omega_1+\omega_2))$ is Stein fillable. 
We will see here this claim for the case $k=2$ (see \cite[Section~2.2.2]{KO} for the case $k=1$). 
The symplectic manifold $(S^2\times S^2, \omega_1+\omega_2)$ can be symplectically embedded into $(\CP^3, \omega_{\FS})$ as the quadric $$Q=\{(z_0: x_1: z_2 :z_3) \in \CP^3 \mid z_0^2+z_1^2+z_2^2+z_3^2=0\}.$$ 
It is straightforward to check that with the standard complex structure $J$ on $\CP^3$, the tuple $(\CP^3, \omega_{\FS}, J; Q)$ is a smoothly polarized K\"ahler manifold of degree $2$ in the sense of \cite[Definitions 2.3.A]{Biran} (see also \cite[Section 3.1.2]{Biran}).
This shows that for some tubular neighborhood $\nu_Q$ of $Q$ in $\CP^3$, the boundary of the complement $\CP^3 \setminus \nu_Q$ carries a contact structure, with which the boundary is contactomorphic to $(M_2,\xi_2)$. 
Moreover, the complement $(\CP^3 \setminus \nu_{Q}, J|_{\CP^3 \setminus \nu_{Q}})$ serves as a Stein filling of this contact manifold. 
This example implies that even if the cohomology class of a symplectic form is non-primitive, the associated Boothby--Wang bundle may have a Stein fillable contact structure. 

Notice that the condition $k> \int_{\Sigma} \omega^{n}$ in Theorem \ref{thm: main} is sharp for the symplectic manifold $(S^2 \times S^2, \omega_1+\omega_2)$ since 
$$
	\int_{S^2 \times S^2}(\omega_1+\omega_2)^2=2, 
$$ 
and the Boothby--Wang contact manifold $(M_k, \xi_k)$, with $k=1,2$, is Stein fillable as observed above. 
\end{example}

\begin{remark}\label{rmk: verification}
Here we shall verify that Theorem~\ref{thm: main} is consistent with a theorem of Giroux, who enhanced the result of Donaldson \cite[Theorem~1]{Don} to \cite[Theorem~1]{Giroux_Remarks} as follows: given a closed integral symplectic manifold $(X, \Omega)$ of dimension $2n+2$, for any sufficiently large integer $k>0$, the Poincar\'e dual $k[\Omega] \in H^2{(X; \Z)}$ can be represented by a symplectic submanifold $\Sigma \subset X$, and the complement $X \setminus \Sigma$ admits a Stein structure. 
By \cite[Proposition~5]{Giroux_Remarks}, the complement $X \setminus \nu_\Sigma$ of a tubular neighborhood $\nu_\Sigma$ of $\Sigma$ has contact boundary, which is contactomorphic to the Boothby--Wang contact manifold $(M_k,\xi_k)$ associated with $(\Sigma, k\Omega|_{\Sigma})$. 
Moreover, similarly to Example \ref{rem: k}, this complement serves as a Stein filling of $(M_k,\xi_k)$; hence, it is Stein fillable. 

The integer $k$ in \cite[Theorem~1]{Giroux_Remarks} does not satisfy the assumption of Theorem~\ref{thm: main}, even if it is extremely large. 
In fact, we have 
$$
	\int_{\Sigma} (\Omega|_{\Sigma})^{n}=[\Omega]^{n} \smile \PD[\Sigma]=k\int_{X}\Omega^{n+1} \geq k, 
$$
where the last inequality follows from the fact that $[\Omega]$ is a second cohomology class with \textit{integer} coefficients. 
Thus, the consistency of the theorem has been verified. 
\end{remark}


\subsection{Boothby--Wang orbibundles}\label{section: orbibundle}

Theorem~\ref{thm: main} generalizes to Boothbby--Wang orbibundles. 
We refer the reader to \cite{BG} and \cite{ALR} for orbifolds, orbibundles and orbifold cohomology; see also \cite{KL} for Boothby--Wang orbibundles. 
In what follows, an orbifold always means a smooth effective orbifold. 

\begin{theorem}\label{thm: orbibundle}
Let $(\Sigma, \omega)$ be a closed integral symplectic orbifold of dimension $2n \geq 4$, and also let $k$ be an integer such that $k > \int_\Sigma \omega^{n}$ and the Boothby--Wang orbibundle $M_k$ over the symplectic orbifold $(\Sigma, k\omega)$ is a manifold. 
Then, $M_k$ carries no Stein fillable contact structure. 
\end{theorem}

\begin{proof}
The proof is similar to Theorem~\ref{thm: main} and Proposition~\ref{prop}. 
It suffices to show that $\pi_k^*[\omega^n] \neq 0$ and that $\pi^*_k[\omega]$ is a torsion element of $H^2(M_k;\Z)$ for the given orbibundle $\pi_k \colon M_k \rightarrow \Sigma$.

The Gysin exact sequence also holds for the orbibundle $\pi_k \colon M_k \rightarrow \Sigma$ with orbifold cohomology. 
Thus, replace the exact sequences (\ref{gysin}) and (\ref{gysin2}) by 
\begin{align}\label{gysin3}
	\cdots \rightarrow H_{\mathrm{orb}}^{2n-2}(\Sigma; \Z) \xrightarrow{k[\omega] \smile} H_{\mathrm{orb}}^{2n}(\Sigma; \Z) \xrightarrow{\pi_k^*} H_{\mathrm{orb}}^{2n}(M_k;\Z) \rightarrow \cdots
\end{align}
and 
\begin{align}\label{gysin4}
	H_{\mathrm{orb}}^{0}(\Sigma; \Z) \xrightarrow{k[\omega] \smile} H_{\mathrm{orb}}^{2}(\Sigma; \Z) \xrightarrow{\pi_k^*} H_{\mathrm{orb}}^{2}(M_k;\Z) \rightarrow \cdots, 
\end{align}
respectively. 
Note that $H_{\mathrm{orb}}^{*}(M_k; \Z) \cong H^{*}(M_k;\Z)$ since $M_k$ is a manifold. 

In view of (\ref{gysin4}), it is straightforward to check that $k(\pi_k^*[\omega])=0$. 
To prove $\pi^*_k[\omega^n] \neq 0$, consider the natural homomorphism 
$$\phi \colon H_{\orb}^{2n}(\Sigma; \Z) \rightarrow H_{\orb}^{2n}(\Sigma; \R).$$ 
Combining \cite[Corollary~4.3.8]{BG} with the de Rham theorem for orbifolds yields   
$$
	H_{\orb}^{2n}(\Sigma; \R)  \cong H^{2n}(\Sigma;\R) \cong H_{\mathrm{dR}}^{2n}(\Sigma). 
$$
By the Poincar\'e duality \cite[p.~35]{ALR}, the last de Rham cohomology group is isomorphic to $\R$ via $[\eta] \mapsto \int_{\Sigma} \eta$. 
Under these identifications, the image $\im(\phi)$ of $\phi$ lies in $\Z \subset \R$, and hence the image $\im( \phi \circ (k[\omega] \smile) )$ of the composition $\phi \circ (k[\omega] \smile)$ lies in $k\Z \subset \R$. 
Since $\omega$ is a symplectic form on the orbifold $\Sigma$, we have $\phi([\omega^n]) \neq 0 \in H_{\orb}^{2n}(\Sigma;\R)$. 
Thus, for $k > \int_\Sigma \omega^n$, we conclude that $\phi([\omega^n]) \not\in \im( \phi \circ (k[\omega] \smile) )$ and $[\omega^n] \not\in \im(k[\omega] \smile)$, which implies that $\pi_k^*[\omega^n] \neq 0$. 
\end{proof}

\begin{remark}
Given a closed integral symplectic orbifold $(\Sigma, \omega)$ of dimension $2n$, necessary and sufficient conditions under which the Boothby--Wang orbibundle over  $(\Sigma, \omega)$ is a manifold are known. 
For instance, according to \cite[Theorem~1.3]{KL}, one such condition is that the homomorphism 
$$[\omega] \smile \colon H^{i}_{\orb}(\Sigma; \Z) \rightarrow H^{i+2}_{\orb}(\Sigma; \Z)$$ 
is an isomorphism for every $i  \geq 2n+1$. 
\end{remark}


\section{Application to smoothability of singularities}\label{section: singularity}

In this short section, we give an application of the main theorem to singularity theory. 

We first briefly recall some basic material of singularity theory. (For more details see e.g. \cite{GLS}, \cite{Greuel} or \cite{PP_winter}.)  
A \textit{singularity} $(X, x)$ is a complex space germ, and it is said to be \textit{smoothable} if there exists a $1$-dimensional deformation $\phi \colon (\mathcal{X}, x) \rightarrow (\C,0)$ of $(X,x)$, called a \textit{smoothing} of $(X,x)$, such that for every $t \in \C \setminus \{0\}$, the fiber $\mathcal{X}_t=\phi^{-1}(t)$ is smooth. 
Now let $(X, x) \subset (\C^N, \bm{0})$ be an isolated singularity. 
Then there exists $\epsilon_0>0$ such that the intersection of $(X,x)$ with the sphere $S^{2N-1}_\epsilon \subset \C^N$ of radius $\epsilon \in (0, \epsilon_0]$ becomes a smooth manifold, and its diffeomorphism type is independent of the choice of $\epsilon$; this intersection is called the \textit{link} of the singularity $(X,x)$. 
It inherits a contact structure $\xi$ from the standard contact structure on $S^{2N-1}_\epsilon$.
Suppose that an isolated singularity $(X, x)$ is smoothable. 
Then, by \cite[Proposition 1.5]{GLS}, a smoothing $\phi \colon (\mathcal{X},x) \rightarrow (\C,0)$ factors as 
$$
	(\mathcal{X}, x) \xhookrightarrow{i} (\C^{N}, \bm{0}) \times (\C, 0) \xrightarrow{\mathrm{pr_2}} (\C,0),
$$
where $i$ is a closed embedding and $\mathrm{pr}_2$ the second projection. 
For any $t \in \C$ sufficiently close to $0$, the intersection 
$$
	\mathcal{X}_t \cap (S_{\epsilon}^{2N-1} \times \{t\}) 
$$
is regarded as a contact manifold similarly to the above, which is contactomorphic to $(X \cap S^{2N-1}_\epsilon, \xi)$ by the Gray stability theorem. 
Moreover, the Milnor fiber $\phi^{-1}(t) \cap D^{2N}_{\epsilon}$, with $t \neq 0$, serves as a Stein filling of this contact manifold. 
Thus, $(X \cap S^{2N-1}_\epsilon, \xi)$ is Stein fillable if $(X,x)$ is smoothable. 
Note that examples of non-smoothable singularities are known; see e.g. \cite[pp.~418--420]{Greuel} and the references therein. 

Singularities we are interested in are derived from negative line bundles over projective manifolds. 
Let $\Sigma$ be a projective manifold and $\pi \colon L \rightarrow \Sigma$ a negative holomorphic line bundle, that is, a holomorphic line bundle with $c_1(L)=-[\omega]$, where $\omega$ is a K\"ahler form on $\Sigma$. 
According to the result of Grauert \cite[Satz 5 on p.~350]{Grauert62}, contracting the zero-section of $L$ yields an affine variety $(X, \bm{0}) \subset (\C^N, \bm{0})$ with an isolated singular point at the origin. 
The question is whether this singularity $(X, \bm{0})$ is smoothable or not. 
In his paper \cite[Section~4]{PP}, Popescu-Pampu established an obstruction to smoothability. 
Here we present a different obstruction to smoothability. 

\begin{proposition}
Let  $\pi \colon L \rightarrow M$ be a negative holomorphic line bundle over a projective manifold $\Sigma$ of complex dimension $\geq 2$ and $(X^{(k)}, \bm{0})$ the isolated singularity obtained by contracting the zero-section of $\pi^{\otimes k} \colon L^{\otimes k} \rightarrow \Sigma$ for $k \in \N$. 
Then, for every integer $k> \langle (-c_1(L))^{\dim_\C \Sigma}, [\Sigma] \rangle$, 
the link of the singularity $(X^{(k)}, \bm{0})$ with the canonical contact structure $\xi$ is not Stein fillable. 
In particular, the singularity $(X^{(k)}, \bm{0})$ is not smoothable. 
\end{proposition}

\begin{proof}
By construction, the contraction $\psi \colon L^{\otimes k} \rightarrow X^{(k)} \subset \C^{N}$ is biholomorphic away from the zero-section, and the image of each fiber of $\pi^{\otimes k} \colon L^{\otimes k} \rightarrow \Sigma$ is a complex line in $\C^{N}$ passing through the origin. 
Therefore, the restriction of the standard inner product on $\C^{N}$ to $\im \psi$ defines a hermitian metric on $L^{\otimes k}$, and the link of the singularity $(X^{(k)}, \bm{0})$ is diffeomorphic to the principal circle bundle $M(L^{\otimes k}) \rightarrow \Sigma$ associated with $\pi^{\otimes k} \colon L^{\otimes k} \rightarrow \Sigma$. 
This circle bundle can be regarded as the Boothby--Wang bundle over $(\Sigma, k\omega)$, where $\omega$ is a K\"ahler form on $\Sigma$ satisfying $c_1(L)=-[\omega]$. 
Thus, by Proposition \ref{prop}, if $k> \int_{\Sigma}\omega^{\dim_\C \Sigma}=\langle (-c_1(L))^{\dim_\C \Sigma}, [\Sigma] \rangle$, the manifold $M(L^{\otimes k})$ does not bound a compact manifold with the homotopy type of a CW complex of dimension $ \leq \dim_{\C} \Sigma+1$. 
Hence, the link of the singularity $(X^{(k)}, \bm{0})$ is not Stein fillable in this case. 
The last assertion follows immediately from the relationship between smoothability and Stein fillability observed above. 
\end{proof}

\subsection*{Acknowledgements}
The author would like to thank Yakov Eliashberg for useful comments on his question \cite[p.~4]{CourteAIM} and Zhengyi Zhou for helpful comments on the first draft, especially pointing out that Theorem \ref{thm: main} generalizes to orbibundles.
This work was supported by JSPS KAKENHI Grant Numbers 20K22306, 22K13913.


\end{document}